\documentclass[12pt]{amsart}
\usepackage[margin=1.2in]{geometry}
\usepackage{graphicx, amsmath, amsthm, amssymb, xcolor, hyperref} 

\theoremstyle{plain}
\newtheorem{theorem}{Theorem}[section]
\newtheorem{lemma}[theorem]{Lemma}
\newtheorem{proposition}[theorem]{Proposition}
\newtheorem{corollary}[theorem]{Corollary}

\theoremstyle{definition}

\newtheorem{remark}[theorem]{Remark}

\newtheorem{example}[theorem]{Example}

\numberwithin{equation}{section}
\DeclareMathOperator{\supp}{supp}

\newcommand{\R}{\mathbb{R}}
\newcommand{\C}{\mathbb{C}}
\newcommand{\N}{\mathbb{N}}
\newcommand{\dt}{\text{d}t}
\newcommand{\abs}[1]{\lvert#1\rvert}
\newcommand{\norm}[1]{\lVert#1\rVert}
\newcommand{\floor}[1]{\lfloor#1\rfloor}

\renewcommand{\H}[1]{{\mathcal{D}^{\operatorname{ext}}_{\infty, #1}(\lambda)}}

\title[Dirichlet series generating holomorphic functions of finite order]{Riesz summability of Dirichlet series generating holomorphic functions of finite order}
\date{}

\author[A. Debrouwere]{Andreas Debrouwere}
\address{A. Debrouwere, Department of Mathematics and Data Science \\ Vrije Universiteit Brussel, Belgium\\ Pleinlaan 2 \\ 1050 Brussels \\ Belgium}
\email{andreas.debrouwere@vub.be}

\author[Y. Tranoy]{Yarne Tranoy}
\address{Y. Tranoy Department of Mathematics and Data Science \\ Vrije Universiteit Brussel, Belgium\\ Pleinlaan 2 \\ 1050 Brussels \\ Belgium}
\email{yarne.tranoy@vub.be}



\begin{document}

\maketitle

\begin{abstract}
Given a frequency $\lambda$, we study the Riesz summability of $\lambda$-Dirichlet series $\sum_{n=1}^\infty a_n e^{-\lambda_n s}$ generating holomorphic functions of finite order. We present a new separation condition on the frequency $\lambda$ ensuring that, for any $k \geq 0$, each $\lambda$-Dirichlet series that is somewhere Riesz summable of some order and admits a holomorphic extension $f$ to the right half-plane $\C_0$ satisfying $f(s) = O(|s|^k)$ as $|s| \to \infty$  on $\C_0$, is  in fact Riesz summable of order $k$ on $\C_0$. This extends Bohr's theorem, which corresponds to the case $k = 0$.   Our work improves a recent result of Defant and Schoolmann, who showed the above property under Landau's condition (LC).  Along the way, we also establish novel  bounds on the coefficients of such $\lambda$-Dirichlet series and, under a mild condition on the frequency $\lambda$, show that they are optimal.
\end{abstract}

\section{Introduction}

By a \emph{frequency}, we mean an increasing sequence $\lambda=(\lambda_n)_{n\in\N}$ of non-negative numbers tending to infinity.  A series $\sum_{n = 1}^\infty a_n e^{-\lambda_n s}$, where $a_n \in \C$, $n \in \N$, and $s \in \C$ denotes a complex variable, is called a \emph{$\lambda$-Dirichlet series}. We denote by $\mathcal{D}(\lambda)$ the space of all $\lambda$-Dirichlet series. This general framework includes as important examples the space $\mathcal{D}(\log n)$ of ordinary Dirichlet series and the space $\mathcal{D}(n)$, which, up to a change of variables, coincides with the space of power series.

This article is devoted to the study of Riesz summability properties of $\lambda$-Dirichlet series generating holomorphic functions of finite order, a topic which goes back to Hardy and Riesz \cite{hardy}. Most importantly, following recent work of Defant and Schoolmann \cite{finorder, bohrpoly}, we investigate a version of Bohr's theorem in this setting.

We start by discussing Bohr's original theorem. Given $\theta \in \R$, we write $\C_\theta = \{ s \in \C \mid \operatorname{Re} s > \theta\}$. Bohr \cite{Bohr} showed that, if a frequency $\lambda$ satisfies the separation condition
\begin{equation}
\tag{BC}
    \exists l,C>0 \, \forall n\in\N : \lambda_{n+1}-\lambda_n \geq C e^{-l\lambda_n},
\end{equation}
every $\lambda$-Dirichlet series  that converges on some half-plane and admits a bounded holomorphic extension to  $\C_0$ converges uniformly on $\C_\varepsilon$ for every $\varepsilon >0$. We will say that a frequency $\lambda$ satisfies \emph{Bohr's theorem} if the latter convergence property holds true. Note that $(\log n)$ has (BC) and thus satisfies Bohr's theorem.

Later on, Landau \cite{Landau} proved that the weaker condition
\begin{equation}
\tag{LC}
    \forall\delta>0 \, \exists C>0 \, \forall n\in\N : \lambda_{n+1}-\lambda_n\geq C e^{-e^{\delta\lambda_n}}
\end{equation}
suffices for $\lambda$ to satisfy Bohr's theorem.  
On the negative side, Neder \cite{Neder}  constructed, given any $l >0$, a  frequency $\lambda$ satisfying 
$$
\exists C>0 \, \forall n\in\N : \lambda_{n+1}-\lambda_n\geq C e^{-e^{l\lambda_n}}
$$
that does not satisfy Bohr's theorem.  

Over the past 20 years, there has been  renewed interest in Bohr's theorem. In \cite{BCQ}, a quantitative version of Bohr's theorem, in the the form of a maximal inequality, was established for ordinary Dirichlet series. This was extended to general $\lambda$-Dirichlet series satisfying (BC) and (LC) in \cite{schoolmann}. Bayart \cite{bayart}  proved a quantitative Bohr's theorem for frequencies $\lambda$ that satisfy the condition 
\begin{equation}
\tag{NC} 
\forall \delta>0 \, \exists C>0 \, \forall n \in \N \, \exists m> n:  \log\left(\frac{\lambda_m+\lambda_n}{\lambda_m-\lambda_n}\right)+(m-n)\leq C e^{\delta\lambda_n}.
\end{equation}
By \cite[Lemma 2.1]{bayart}, (LC) implies (NC). Conditions (LC) and (BC) impose that consecutive gaps $\lambda_{n+1} - \lambda_n$ cannot decay too fast, whereas (NC) allows consideration of wider gaps $\lambda_{m} - \lambda_n$, $m >n$. Bayart's proof  is based on the so-called Saksman vertical convolution formula \cite[Theorem 2.4]{bayart}, which is a simple but very flexible substitute of Perron's formula, together with suitable mollifiers; see \cite{first_saksman} for  the origin of this formula.

In their classical monograph \cite{hardy}, Hardy and Riesz made a systematic study of Dirichlet series generating holomorphic functions of finite order from the point of view of Riesz summability. We refer to the article \cite{finorder} of Defant and Schoolmann for a modern treatment of this theory. 

We need to introduce some terminology before we can continue. Let $\lambda$ be a frequency and $k \geq 0$. The \emph{$(\lambda, k)$-Riesz means} of an element $D = \sum_{n =1}^\infty a_n e^{-\lambda_n s} \in \mathcal{D}(\lambda)$ are defined as 
$$
R^{\lambda,k}_x(D)(s) = \sum_{\lambda_n\leq x}a_n \left(1-\frac{\lambda_n}{x}\right)^k e^{-\lambda_n s}, \qquad x > 0.
$$
The series $D$ is said to be \emph{$(\lambda,k)$-summable} at $s \in \C$ if the limit
$$
\lim_{x \to \infty} R^{\lambda,k}_x(D)(s)
$$
exists. We define $\H k$ as the space consisting all $D \in \mathcal{D}(\lambda)$ such that $D$ is $(\lambda,l)$-summable on $\C_\theta$ for some $l,\theta >0$ and its limit function
$$
f(s) = \lim_{x \to \infty}R^{\lambda,l}_x(D)(s), \qquad s \in \C_{\theta},
$$
has a holomorphic extension to $\C_0$, which we still denote by $f$, satisfying
$$
 \sup_{s \in \C_0} \frac{|f(s)|}{(1+|s|)^k} < \infty.
$$
The space $\H k$ can be be naturally identified with the space  $H^{\lambda}_{{\infty, k}}[\operatorname{Re} >0]$ considered in \cite{finorder, bohrpoly}; see Remark \ref{remid}. 
 
In this terminology, a fundamental result of Hardy and Riesz \cite[Theorem 41]{hardy} asserts that, for any frequency $\lambda$,  each $D \in \H k$ is $(\lambda,l)$-summable on $\C_0$ for every $l >k$; see also \cite[Section 3.5]{finorder} and \cite[Section 5.1]{helson}. It is now natural to ask whether every $D \in \H k$ is in fact $(\lambda,k)$-summable on $\C_0$. Due to its analogy with Bohr's theorem, we say that $\lambda$ satisfies \emph{Bohr's theorem of order $k$} if this holds. This property was recently introduced  and studied by Defant and Schoolmann \cite{finorder, bohrpoly}.  We note that Bohr's theorem of order $0$ is equivalent to Bohr's theorem \cite[Remark 2.7]{bohrpoly}.

Defant and Schoolmann proved Bohr's theorem of order $k$ under the condition (BC) \cite[Theorem 3.26]{finorder} and later relaxed this assumption to (LC) \cite[Theorem 3.8]{bohrpoly}. The main goal of this article is to improve upon these results by establishing Bohr's theorem of order $k$ for frequencies $\lambda$ satisfying a strictly weaker condition than (LC). This new condition is a strengthened variant of (NC); see Subsection \ref{subs} for its definition.  As in \cite{bohrpoly}, we show quantitative versions of Bohr's theorem of order $k$ via maximal inequalities. 


The proof that (LC) implies Bohr's theorem of order $k$ in \cite{bohrpoly} is based on a careful application of a technique due to Hardy and 
Riesz \cite[Theorem 22]{hardy}. Our approach is mostly different. For integral orders,  where the core of our argument lies, it is inspired by Bayart's method \cite{bayart}: We rely on a generalization of the Saksman vertical convolution formula adapted to the spaces $\H k$. As in \cite{bohrpoly}, we reduce the general (non-integral) case to the integral one by means of an argument from \cite[Theorem 22]{hardy}.

To extend Bayart's approach to our setting, we need to construct suitable mollifiers related to the Riesz kernels $\left(1-\frac{\cdot}{x}\right)^k$, $k \in \N$. This constitutes a substantial part of our work. Another difficulty is that the coefficients of elements of $\H k$, $k>0$, are in general unbounded, unlike in the case $k =0$. We establish new coefficient bounds for the  elements of $\H k$, which improve those obtained in \cite[Theorem 3.12]{finorder}.  Moreover, for 
$k \leq 1$ or under a mild condition on the frequency $\lambda$, we show that they are optimal. 

This article is organized as follows. In the preliminary Section \ref{sect-1}, we collect several basic definitions and results about Riesz summability and Dirichlet spaces. The generalization of the Saksman vertical convolution formula is shown in Section \ref{sect-2}. As a first application of this formula, we give in Section \ref{sect-3}  a different proof of a maximal inequality related to the $(\lambda,l)$-summability of elements of $\H k$, where $l >k$; originally shown in \cite[Theorem 3.6]{finorder} and \cite[Remark 3.7]{bohrpoly}. Next, in Section \ref{sect-4}, we study coefficient bounds for the elements of $\H k$. Section \ref{sect-5} is devoted to proving our main result, in which we establish Bohr's theorem of order $k$ for a new class of frequencies. 

\section{Preliminaries}\label{sect-1}
In this preliminary section, we recall the definition and some basic properties of Riesz summability and Dirichlet series. All the results are due to Hardy and Riesz \cite{hardy}; see \cite{finorder} for detailed proofs of these statements. Furthermore, we introduce the main spaces of Dirichlet series studied in this article, originally defined in \cite{finorder, bohrpoly}. We write $\N = \{1,2, \ldots\}$ and $\N_0 = \N \cup \{ 0\}$.

Let $\lambda$ be a frequency and $k  \geq 0$. Given a (formal) series $C = \sum_{n = 1}^\infty a_n$ of complex numbers, we define the \emph{$(\lambda, k)$-Riesz means} of $C$ as
$$
R^{\lambda,k}_x(C) = \sum_{\lambda_n\leq x}a_n \left(1-\frac{\lambda_n}{x}\right)^k, \qquad x > 0.
$$
The series $C$ is said to be \emph{$(\lambda, k)$-Riesz summable} if the limit
$$
\lim_{x \to \infty} R^{\lambda,k}_x(C) 
$$
exists. We will also consider the summatory function
\[S_x^{\lambda, k}(C)=x^kR_x^{\lambda, k}(C)=\sum_{\lambda_n\leq x}a_n \left(x-\lambda_n\right)^k, \qquad x > 0.\]

We have the following consistency result.
\begin{proposition} \cite[Theorem 16]{hardy} \emph{(see also \cite[Theorem 2.7]{finorder})}
\label{con-order}
Let $\lambda$ be a frequency and   $l \geq k \geq 0$. If a series $C = \sum_{n = 1}^\infty a_n$ is $(\lambda,k)$-summable, then it is also  $(\lambda,l)$-summable and 
$$
\lim_{x \to \infty} R^{\lambda,k}_x(C)  = \lim_{x \to \infty} R^{\lambda,l}_x(C).
$$
\end{proposition}
The following two lemmas allow for changing the order of Riesz summability.
\begin{lemma}  \cite[Lemma 6]{hardy} \emph{(see also \cite[Lemma 4.4]{finorder})}
\label{technical-0}
Let $\lambda$ be a frequency, $k \geq 0$, and $\varepsilon >0$.  Let $C = \sum_{n = 1}^\infty a_n$ be a series. Then, for all $x > 0$,
    \[ S^{\lambda, k + \varepsilon}_x(C) =  \frac{\Gamma(k +\varepsilon+1)}{\Gamma(k+1)\Gamma(\varepsilon)}\int_0^x S_y^{\lambda, k} (C)(x-y)^{\varepsilon-1}{\rm d} y.\]
\end{lemma}

\begin{lemma}  \cite[Lemma 8]{hardy} \emph{(see also \cite[Lemma 4.9]{finorder})}
\label{technical}
Let $\lambda$ be a frequency,   $k \geq 0$, and $0<\varepsilon \leq1$.  Let $C = \sum_{n = 1}^\infty a_n$ be a series. Then, for all $0<\xi \leq x$,
    \[\frac{\Gamma(k +\varepsilon+1)}{\Gamma(k+1)\Gamma(\varepsilon)}\abs{\int_0^\xi S_y^{\lambda, k} (C)(x-y)^{\varepsilon-1}{\rm d} y}\leq 2\sup_{y \leq \xi}\abs{S_y^{\lambda, k+\varepsilon}(C)}.\]
\end{lemma}

Next, we consider Dirichlet series. Let $\lambda$ be a frequency. We write $a_n(D) = a_n$ for $D =  \sum_{n = 1}^\infty a_n e^{-\lambda_n s} \in \mathcal{D}(\lambda)$. Let $k \geq 0$. As in the introduction, we define, for $D \in \mathcal{D}(\lambda)$ and $s \in \C$,
\[R^{\lambda,k}_x(D)(s) =  R^{\lambda,k}_x\left (\sum_{n = 1}^\infty a_n(D) e^{-\lambda_n s} \right) = \sum_{\lambda_n\leq x}a_n(D) \left(1-\frac{\lambda_n}{x}\right)^ke^{-\lambda_n s}, \qquad x > 0; \] 
$S^{\lambda,k}_x(D)(s)$ is defined similarly.
We say that $D$ is \emph{$(\lambda,k)$-summable at $s \in \C$} if the series $\sum_{n = 1}^\infty a_n(D) e^{-\lambda_n s}$ is $(\lambda,k)$-summable, that is, if the limit
$$
 \lim_{x \to \infty}R^{\lambda,k}_x(D)(s)
$$
exists. 
Given $\theta  \in \R$, we write $\C_\theta = \{ s \in \C \mid \operatorname{Re} s > \theta\}$. We have the following basic result about the Riesz summability of $\lambda$-Dirichlet series.
\begin{proposition} \cite[Theorem 23 and 24]{hardy} \emph{(see also \cite[Theorem 2.9 and Remark 2.11]{finorder})}
\label{consistent}
Let $\lambda$ be a frequency and  $k \geq 0$. Let $D \in \mathcal{D}(\lambda)$ be $(\lambda, k)$-summable at $s_0  \in \C$. Set $ \sigma_0 = \operatorname{Re} s_0$. Then, $D$ is $(\lambda, k)$-summable at every $s \in \C_{\sigma_0 }$. Moreover, the limit function
$$
f(s) = \lim_{x \to \infty}R^{\lambda,k}_x(D)(s), \qquad s \in \C_{\sigma_0},
$$
 is holomorphic on $\C_{\sigma_0}$ and
    \[\Gamma(k+1)\frac{f(s_0 +s)}{s^{k+1}}=\int_0^\infty e^{-sx}S_x^{\lambda, k}(D)(s_0) {\rm d}x, \qquad s \in \C_0.\]
\end{proposition}
Let $\lambda$ be a frequency and  $k \geq 0$. We define the \emph{abscissa of pointwise $(\lambda,k)$-summability} of an element $D \in \mathcal{D}(\lambda)$ as 
$$
\sigma^{\lambda, k}(D) = \inf\{ \sigma\in\R \mid D \text{ is } (\lambda, k)\text{-Riesz summable at } \sigma\}.
$$
By Proposition \ref{consistent}, $D$ is pointwise $(\lambda,k)$-summable on $\C_{\sigma^{\lambda, k}(D)}$ and $D$ is not $(\lambda,k)$-summable at every point of $\C \backslash \overline{\C_{\sigma^{\lambda, k}(D)}}$. The following  Bohr-Cahen type formula is very useful to determine this abscissa.
\begin{proposition} \cite[Theorem 31]{hardy} \emph{(see also \cite[Theorem 2.16]{finorder})}
    \label{Bohr-Cahen} 
    Let $\lambda$ be a frequency and  $k \geq 0$.  For every $D \in \mathcal{D}(\lambda)$, it holds that    
    \[\sigma_c^{\lambda, k}(D) \leq \limsup_{x\to \infty} \frac{\log\abs{R_x^{\lambda, k}(D)(0)}}{x}.\]
\end{proposition}

Let $\lambda$ be a frequency and  $k \geq 0$.
Recall from the introduction that  $\H{k}$ denotes the space consisting of all $D \in \mathcal{D}(\lambda)$ for which there is $l \geq0$ such that  $\sigma^{\lambda,l}(D) < \infty$ and the limit function
$$
f(s) = \lim_{x \to \infty}R^{\lambda,l}_x(D)(s), \qquad s \in \C_{\sigma^{\lambda,l}(D) },
$$
has a (unique)  holomorphic extension to $\C_0$, which we still denote by $f$, satisfying
$$
\| f \|_{k} = \sup_{s \in \C_0} \frac{|f(s)|}{(1+|s|)^k} < \infty.
$$
By Proposition \ref{con-order}, the function $f$ is independent of $l$. Henceforth, we write $f = f_D$ and 
$\| D \|_{\H k} = \| f_D \|_{k}$.

\begin{remark}\label{remid}
Let $\lambda$ be a frequency and $k \geq 0$. Following \cite{finorder, bohrpoly}, we define $H^{\lambda}_{\infty,k}[\operatorname{Re} >0]$ as the space consisting of all holomorphic functions $f$ on $\C_0$ satisfying $\| f\|_k < \infty$ and for which there exists $D \in \mathcal{D}(\lambda)$ and $l \geq 0$ such that $\sigma^{\lambda,l}(D) < \infty$  and
$$
f(s) = \lim_{x \to \infty}R^{\lambda,l}_x(D)(s), \qquad s \in \C_{\sigma^{\lambda,l}(D) }.
$$
Then,
\begin{equation}
\label{isometry}
\H k \to (H^{\lambda}_{\infty,k}[\operatorname{Re} >0], \| \cdot \|_k), \, D \mapsto f_D,
\end{equation}
is an isometric linear surjection. Later on, in Remark \ref{Banach}(i), we will see that \eqref{isometry} is in fact an isometric isomorphism, or, what amounts to the same, that $\| \cdot \|_{\H k}$ is a norm on $\H k$. Hence, we are actually  working with the same spaces as in \cite{finorder, bohrpoly}; the only difference is that we take the Dirichlet series themselves, rather than their limiting functions, as the primary objects. We believe this point of view is more natural for our purposes.
\end{remark}

\section{A generalization of the Saksman vertical convolution formula}\label{sect-2}
In this section, we prove a  generalization of the Saksman vertical convolution formula \cite[Theorem 2.4]{bayart} adapted to the spaces $\H k$; see Helson's general summability theorem \cite[Chapter 4]{helson} for related work. This will be one of our main tools in the rest of this paper. We fix the constants as follows in the Fourier transform
$$ 
\mathcal{F}(\varphi)(t) = \widehat{\varphi}(t) =  \int_{-\infty}^\infty \varphi(x) e^{-it x} {\rm d}x, \qquad \varphi \in L^1(\R).
$$

\begin{proposition}\label{saksmanth}
Let $\lambda$ be a frequency and $k \geq 0$. Let  $\varphi \in C_c(\R)$ be such that
 \begin{equation}
  \label{saksman-cond}
    \int_{-\infty}^\infty |\widehat{\varphi}(t)|(1+\abs{t})^k {\rm d}t<\infty.
     \end{equation}
For every  $D \in \H k$, it holds that
  \begin{equation}
  \label{saksman}
  \sum_{n=1}^\infty a_n(D) \varphi(\lambda_n) e^{-\lambda_n s}= \frac{1}{2\pi}\int_{-\infty}^{\infty} f_D(s+it)\widehat{\varphi}(-t) {\rm d}t, \qquad s \in \C_0.
  \end{equation}
\end{proposition}
\begin{proof}
Fix $D  \in \H k$ and pick $l \geq 0$ such that $\sigma^{\lambda,l}(D) < \infty$.  In view of Proposition \ref{con-order}, we may assume that $l \in \N$.
We first prove \eqref{saksman} for $\varphi \in \mathcal{D}(\R)$. Since both sides of this identity are holomorphic on $\C_0$, it suffices to show it for $s \in \C_{\sigma^{\lambda,l}(D)}$. Write $s = s_0 + \varepsilon$ with  $s_0 \in \C_{\sigma^{\lambda,l}(D)}$ and $\varepsilon >0$.  Set $\varphi_\varepsilon(x)=e^{-\varepsilon x}\varphi(x)\in \mathcal{D}(\R)$. We find that
$$
\sum_{n=1}^\infty a_n(D) \varphi(\lambda_n) e^{-\lambda_n s} = \sum_{n=1}^\infty a_n(D) \varphi_\varepsilon(\lambda_n) e^{-\lambda_n s_0} = \frac{(-1)^{l+1}}{l!} \int_{-\infty}^{\infty} S^{\lambda,l}_x(D)(s_0)\varphi_\varepsilon^{(l+1)}(x)\text{d}x.
$$
Define $S^{\lambda,l}_x(D)(s_0) = 0$ for $x \leq 0$.  Proposition \ref{consistent} yields that
 \[ l !\frac{f_D(s + it)}{(\varepsilon + it)^{l+1}} = l!\frac{f_D(s_0 + \varepsilon + it)}{(\varepsilon + it)^{l+1}}= \mathcal{F}_x(e^{-\varepsilon x}S^{\lambda,l}_{x}(D)(s_0))(t), \qquad  t \in \R.\]    
 Note that
 $$
  \mathcal{F}^{-1}_x(e^{\varepsilon x }\varphi_\varepsilon^{(l+1)}(x))(t)  = (-1)^{l+1}(\varepsilon + it)^{l+1} \mathcal{F}^{-1}(\varphi)(t), \qquad t \in \R.
 $$
 Therefore, by Plancherel's identity,
  \begin{align*}
 \sum_{n=1}^\infty a_n(D) \varphi(\lambda_n) e^{-\lambda_n s} &= \frac{(-1)^{l+1}}{l!} \int_{-\infty}^{\infty} S^{\lambda,l}_x(D)(s_0)\varphi_\varepsilon^{(l+1)}(x)\text{d}x \\
 & =  \frac{(-1)^{l+1}}{l!} \int_{-\infty}^{\infty}  \mathcal{F}_x(e^{-\varepsilon x }S^{\lambda,l}_{x}(D)(s_0))(t)  \mathcal{F}_x^{-1}(e^{\varepsilon x }\varphi_\varepsilon^{(l+1)}(x))(t)\text{d}t \\
&=  \int_{-\infty}^{\infty} f_D(s+it)\mathcal{F}^{-1}(\varphi)(t) {\rm d}t = \frac{1}{2\pi}\int_{-\infty}^{\infty} f_D(s+it)\widehat{\varphi}(-t) {\rm d}t.
\end{align*}
Next, we show the result in general. Take $\varphi \in C_c(\R)$ satisfying \eqref{saksman-cond}. Choose $\psi \in \mathcal{D}(\R)$ with $\int_{-\infty}^\infty \psi(t)\text{d}t=1$ and set $\psi_\varepsilon(x) = \varepsilon^{-1}\psi(x/\varepsilon)$ for $\varepsilon >0$. Then, $\varphi \ast \psi_\varepsilon \in \mathcal{D}(\R)$ for all $\varepsilon >0$. Hence, by what we already have shown,
\[\sum_{n=1}^\infty a_n(D) (\varphi*\psi_\varepsilon)(\lambda_n)e^{-\lambda_n s}=\frac{1}{2\pi}\int_{-\infty}^{\infty} f_D(s+it) \widehat{\varphi}(-t)\widehat{\psi}(-\varepsilon t)\text{d}t, \qquad s \in \C_0.\]
The result now follows by letting $\varepsilon \to 0^+$. 
\end{proof}

\begin{remark}
Bayart's proof of \cite[Theorem 2.4]{bayart} (corresponding to the case $k=0$ above) makes use of the fact that the space of $\lambda$-Dirichlet polynomials is dense in $\H 0$. Our proof is more elementary in that it relies only on the basic Proposition \ref{consistent}; see  also \cite[Remark 2.6]{bayart}.
\end{remark}

\section{ A maximal inequality for higher order Riesz means}\label{sect-3}
As a first application of the Saksman vertical convolution formula, we give in this section a short proof of an explicit maximal inequality for higher order Riesz means, originally shown in \cite[Theorem 3.6]{finorder} and \cite[Remark 3.7]{bohrpoly}. This result will later be used in Section \ref{sect-5} in the proof of our main result.
\begin{theorem}
\label{lkasymp}
Let $\lambda$ be a frequency and $k \geq 0$. For all $l >k$, it holds that
$$
C_{k,l} = \sup_{x >0}\norm{R_x^{\lambda,l}}_{\mathcal{L}(\H k)} < \infty.
$$
Moreover, there is $C >0$ such that, for all $l \in (k,k+1]$,
$$
C_{k,l} \leq \frac{C}{l-k}.
$$
\end{theorem}
In view of  the Borh-Cahen formula (Proposition \ref{Bohr-Cahen}), Theorem \ref{lkasymp} yields the following corollary, originally due to Hardy and Riesz \cite[Theorem 41]{hardy}; see also \cite[Section 5.1]{helson}. 
\begin{corollary}
Let $\lambda$ be a frequency and  $l > k \geq 0$. For every  $D \in \H k$, it holds that $\sigma^{\lambda,l}(D) \leq 0$.
\end{corollary}

We now turn to the proof of Theorem \ref{lkasymp}. The following lemma, essentially due to Helson \cite[p.\ 72, Lemma]{helson}, provides the key step. Since we will need certain bounds that are not explicitly stated there, we include  the proof.
\begin{lemma}\label{lemmahelson}
Let $l > 0$. There exists $\varphi \in C_c(\R)$ satisfying the following properties: 
\begin{itemize}
\item[(i)] $\varphi(x) = (1-x)^l$ for $x \in [0,1]$.
\item[(ii)] $\operatorname{supp} \varphi \subseteq (-\infty,1]$.
\item[(iii)] It holds that
$$
C_l = \sup_{t \in \R}|\widehat{\varphi}(t)|(1+|t|)^{l+1} < \infty. 
$$
\end{itemize}
Moreover, $\displaystyle \sup_{r \in [l,l+1]} C_{r} < \infty$.
\end{lemma}
\begin{proof}
Pick $\psi \in \mathcal{D}((-1,2))$ with $\psi(x) =1$ for $x \in [0,1]$. Then, $\varphi(x) = (1-x)^l_+\psi(1-x) \in C_c(\R)$ satisfies (i) and (ii). We now show (iii). Note that
$$
\int_{0}^\infty x^l e^{-sx} {\rm d} x = \frac{\Gamma(l+1)}{s^{l+1}}, \qquad s >0.
$$
By analytic continuation, this identity in fact holds for all $s \in \C_0$. For $s = 1+it$, $t \in \R$, we find that
$$
\mathcal{F}_x(x^l_+ e^{-x})(t) =  \frac{\Gamma(l+1)}{(1+it)^{l+1}}.
$$
 Set $\kappa(x) = e^x\psi(x) \in \mathcal{D}(\R)$. Then, for all $t \in \R$,
\begin{align*}
|\widehat\varphi(t)| &= |\mathcal{F}_x(x^l_+\psi(x))(-t)| = \frac{1}{2\pi}\abs{(\mathcal{F}_x(x^l_+e^{-x}) \ast \widehat{\kappa})(-t)} \\
&\leq \frac{2^{(l+1)/2}\Gamma(l+1)}{2\pi}\int_{-\infty}^\infty \frac{|\widehat{\kappa}(u+t)|}{(1+|u|)^{l+1}} {\rm d}u \\
&\leq \frac{2^{(l+1)/2}\Gamma(l+1)}{2\pi(1+|t|)^{l+1}}\int_{-\infty}^{\infty}|\widehat{\kappa}(u)|{(1+|u|)^{l+1}} {\rm d}u.
\end{align*}
This proves (iii) and also $\displaystyle \sup_{r \in [l,l+1]} C_{r} < \infty$.
    \end{proof}

\begin{proof}[Proof of Theorem \ref{lkasymp}]  Let $\varphi$ be as in Lemma \ref{lemmahelson} and set $\varphi_x(y) = \varphi(y/x)$ for $x >0$. Then, for all $D \in \H k$,
$$
R_x^{\lambda,l}(D)(s)  = \sum_{n = 1}^\infty a_n(D) \varphi_x(\lambda_n)e^{-\lambda_n s}, \qquad s \in \C_0, x >0.
$$
By the Saksman vertical convolution formula (Proposition \ref{saksmanth}), we find that
\begin{align*}
|R_x^{\lambda,l}(D)(s)| &= \left |  \frac{1}{2\pi}\int_{-\infty}^{\infty} f_D(s+it)\widehat{\varphi}_x(-t) {\rm d}t \right|  \\
&{\leq}   \frac{x}{2\pi} \| D\|_{\H k}(1+|s|)^k \int_{-\infty}^{\infty}(1+|t|)^k |\widehat{\varphi}(-xt)| {\rm d}t   \\
&=   \frac{1}{2\pi} \| D\|_{\H k}(1+|s|)^k\int_{-\infty}^{\infty} \left(1+\frac{|t|}{x}\right)^k |\widehat{\varphi}(t)| {\rm d}t   \\
&\leq   \frac{C_l \max\{1,x^{-k} \}}{2\pi} \| D\|_{\H k}(1+|s|)^k\int_{-\infty}^{\infty}\frac{1}{(1+|t|)^{1+l-k}} {\rm d}t \\
&=   \frac{C_l\max\{1,x^{-k} \}}{\pi(l-k)} \| D\|_{\H k}(1+|s|)^k.
\end{align*}
This implies the result.
\end{proof}

\section{Coefficient bounds}\label{sect-4}
In this section, we apply the Saksman vertical convolution formula to prove new growth estimates for the coefficients of elements of the spaces $\H k$, thereby improving those from \cite[Theorem 3.12]{finorder}.  
Moreover, we prove that these bounds are optimal if $k \leq 1$ or under a mild uniformity condition on the frequency $\lambda$. 

Given a frequency $\lambda$, we define $\varepsilon_1(\lambda)=\min\{1,\lambda_2-\lambda_{1}\}$ and
$$
\varepsilon_n(\lambda)=\min\{1,\lambda_n-\lambda_{n-1}, \lambda_{n+1}-\lambda_n\}, \qquad  n \geq 1.
$$ 
If the frequency $\lambda$ is clear from the context, we will simply write $\varepsilon_n = \varepsilon_n(\lambda)$, $n \in \N$.
\begin{theorem} 
\label{coeff_bounds}
Let $\lambda$ be a frequency and  $k \geq 0$. There exists $C >0$ such that, for all $D \in \H k$,
$$
|a_n(D)| \leq C\norm{D}_{\H k}  \left(\frac{1}{\varepsilon_n}\right)^k, \qquad n \in \N.
$$
\end{theorem}

\begin{proof}
Take $\varphi \in \mathcal{D}((-1,1))$ with $\varphi(0) = 1$. Define $\varphi_n(x)=\varphi\left(\frac{x-\lambda_n}{\varepsilon_n}\right)$ for $n \in \N$. Then, $\varphi_n(\lambda_n) = 1$ and  $\varphi_n(\lambda_m) = 0$ for all $m \neq n$.  By the Saksman vertical convolution formula (Proposition \ref{saksmanth}), we obtain that, for all $D \in \H k$, 
$$
a_n(D)e^{-\lambda_n\sigma} = \frac{1}{2\pi}\int_{-\infty}^{\infty} f_D(\sigma+it)\widehat{\varphi}_n(-t)\mathrm{d}t, \qquad \sigma >0.
$$
We find that
    \begin{align*}
    |a_n(D)| e^{-\lambda_n\sigma} &\leq \frac{1}{2\pi}\int_{-\infty}^{\infty} |f(\sigma+it)| |\widehat{\varphi}_n(-t)|\mathrm{d}t \\
    &\leq  \frac{(1+\sigma)^k \varepsilon_n}{2\pi}\norm{D}_{\H k} \int_{-\infty}^{\infty} (1+\abs{t})^k\abs{\widehat{\varphi}(\varepsilon_nt)}\mathrm{d}t\\  
   &\leq \frac{(1+\sigma)^k}{2\pi}\norm{D}_{\H k} \int_{-\infty}^{\infty} \left(1+\frac{|t|}{\varepsilon_n}\right)^k\abs{\widehat{\varphi}(t)}\mathrm{d}t\\  
      &\leq \frac{(1+\sigma)^k}{2\pi}\norm{D}_{\H k}  \left(\frac{1}{\varepsilon_n}\right)^k \int_{-\infty}^{\infty} \left(1+\abs{t}\right)^k\abs{\widehat{\varphi}(t)}\mathrm{d}t.
    \end{align*}
The result follows by     letting $\sigma \to 0^+$.
\end{proof}

\begin{remark} \label{Banach} 
(i) Theorem \ref{coeff_bounds} implies that  $\| \cdot \|_{\H k}$ is a norm on $\H k$ and thus that \eqref{isometry} is an isometric isomorhpism. \\
(ii) As shown in \cite[Theorem 3.16]{finorder}, Proposition \ref{Bohr-Cahen} together with Theorems \ref{lkasymp} and \ref{coeff_bounds} imply that $\H k$ is a Banach space.
\end{remark}

Under a certain assumption on the order  or the frequency, we can show that the bounds in Theorem \ref{coeff_bounds} are optimal. Our approach, based on the open mapping theorem, is inspired by a method to show optimality results in Tauberian theory; see e.g. \cite{DS}.

\begin{theorem}\label{optimal}
Let $\lambda$ be a frequency and $k \geq 0$. Suppose that one of the two following conditions hold:
\begin{itemize}
\item[(i)] $k \leq 1$.
\item[(ii)] There are $c_1,c_2 >0$ such that, for all $n \geq 2$,  
$$
c_1 \leq \frac{\lambda_{n+1} - \lambda_n}{\lambda_n - \lambda_{n-1}}\leq c_2.
$$
\end{itemize}
Let $(\rho_n)_{n \in \N}$ be a sequence of positive numbers such that, for all $D \in \H k$,
 \begin{equation}
 \label{ass-opt}
 a_n(D)=O(\rho_n).
 \end{equation}
Then,
    \[  \left(\frac{1}{\varepsilon_n}\right)^k = O(\rho_n).\]
\end{theorem}
\begin{proof}
 The assumption \eqref{ass-opt} implies that
$$
\norm{D}=\norm{D}_{\H k}+\sup_{n \in \N}\frac{|a_n(D)|}{\rho_n}, \qquad D \in \H k,
$$
is a norm on $\H k$.  Since $\H k$ is a Banach space (Remark \ref{Banach}(ii)) and $\| \cdot \|$ is stronger than $\| \cdot \|_{\H k}$, the open mapping theorem yields that these norms are equivalent. Consequently, there is $C >0$ such that, for all $D \in \H k$,
 \begin{equation}
 \label{OMT}\sup_{n \in \N} \frac{|a_n(D)|}{\rho_n}\leq C\norm{D}_{\H k}.
 \end{equation}
For $s \in \C_0$, we write $e_s(x) = e^{-xs}$, $x \in \R$. Set $l = \lceil k \rceil$. For $n \in \N$, we define $\lambda$-Dirichlet polynomial
$$
D_n(s) = e_s[\lambda_n, \ldots, \lambda_{n+l}] = \sum_{j =0}^l a_{n,j} e^{-\lambda_{n+j} s},
$$
where
$$
a_{n,j} = \frac{1}{\displaystyle \prod_{\substack{0\leq m\leq l \\ m\neq j}} (\lambda_{n+j}-\lambda_{n+m})},
$$
and we used standard notation for divided differences. If (ii) holds, there are $C_1,C_2 >0$ such that, for all $n \in \N$ and $j = 0, \ldots, l$,
 \begin{equation}
 \label{coeff}
\frac{C_1}{(\lambda_{n+1} - \lambda_n)^l} \leq |a_{n,j} | \leq \frac{C_2}{(\lambda_{n+1} - \lambda_n)^l}.
\end{equation}
In case of (i), the above even holds with  $C_1 = C_2 =1$.
Hence,
$$
|D_n(s)| \leq \frac{C_2(l+1)}{(\lambda_{n+1} - \lambda_n)^l}, \qquad s \in \C_0.
$$
By applying the mean value theorem for divided differences to the real and imaginary part of $D_n$, we find that
$$
|D_n(s)| \leq \frac{2}{l!}|s|^l, \qquad s \in \C_0.
$$
By combining the last two inequalities, we obtain that there is $C_3 >0$ such that, for all $n \in \N$,
$$
\norm{D_n}_{\H k}=\sup_{s \in \C_0} \frac{\abs{D_n(s)}^{1-\frac{k}{l}}\abs{D_n(s)}^{\frac{k}{l}}} {(1+\abs{s})^k} \leq \frac{C_3}{(\lambda_{n+1} - \lambda_n)^{l-k}}.
$$
In view of \eqref{OMT} and \eqref{coeff}, we get that, for all $n \in \N$ and $j = 0, \ldots, l$,
$$
\frac{C_1}{\rho_{n+j}{(\lambda_{n+1} - \lambda_n)^l}}  \leq  \frac{|a_{n,j} |}{\rho_{n+j}} =  \frac{|a_{n+j}(D_n) |}{\rho_{n+j}} \leq C\norm{D_n}_{\H k} \leq   \frac{CC_3}{(\lambda_{n+1} - \lambda_n)^{l-k}}.
$$
Hence,
$$
\frac{1}{(\lambda_{n+1} - \lambda_n)^k}   \leq   \frac{CC_3}{C_1} \rho_{n+j}.
$$
This implies that, for all $n \geq 2$,
$$
\max \left \{ \frac{1}{(\lambda_{n} - \lambda_{n-1})^k},  \frac{1}{(\lambda_{n+1} - \lambda_n)^k} \right\}   \leq   \frac{CC_3}{C_1} \rho_{n}.
$$
Furthermore, by taking $D = e^{-\lambda_n s}$ in \eqref{OMT}, we find that $1 \leq C \rho_n$. This shows the result.
\end{proof}
\begin{remark}
Condition (ii) is satisfied for the frequency $\lambda = (\log n)$, corresponding to the case of ordinary Dirichlet series.
\end{remark}
It is known that the $\eta$-Dirichlet series
$$
\sum_{n =1}^\infty \frac{(-1)^{n+1}}{n^s}
$$
belongs to $\mathcal{D}^{\operatorname{ext}}_{\infty,k}(\log n)$ for $k >1/2$, but not for $k <1/2$. Consequently, for $k > 1/2 > l$,
\begin{equation}
\label{stincl}
\mathcal{D}^{\operatorname{ext}}_{\infty,l}(\log n) \subsetneq \mathcal{D}^{\operatorname{ext}}_{\infty,k}( \log n).
\end{equation}
On the other hand,  $\mathcal{D}^{\operatorname{ext}}_{\infty,k}(n) = \mathcal{D}^{\operatorname{ext}}_{\infty,l}(n)$ for all $l,k \geq 0$ \cite[Proposition 3.3]{finorder}.

For general frequencies $\lambda$ and orders $k > l \geq 0$, it seems unclear  when the spaces $\H k$ and $\H l$ differ. Theorems \ref{coeff_bounds} and \ref{optimal} give the following partial answer to this question:
\begin{corollary}\label{inclusion}
Let $\lambda$ be a frequency such that 
\begin{equation}
\label{spacing}
\inf_{n \in \N} \lambda_{n+1} -\lambda_n = 0.
\end{equation}
Let $k > l \geq 0$ and suppose that $k$ satisfies condition (i) or (ii) from Theorem \ref{optimal}. Then, 
\begin{equation}
\label{stincl2}
\H l \subsetneq \H k.
\end{equation}

\end{corollary}
\begin{proof}
Suppose that $\H l = \H k$. Theorem \ref{coeff_bounds} implies that $a_n(D) =   O \left (\varepsilon^{-l}_n \right)$  for all $D \in \H k$. Hence, by Theorem \ref{optimal}, we obtain that $ \varepsilon^{-k}_n = O \left (\varepsilon^{-l}_n \right)$. As $k > l$, this contradicts \eqref{spacing}.
\end{proof}
\begin{remark}
(i) Corollary \ref{inclusion} for $\lambda = (\log n)$ shows that the strict inclusion \eqref{stincl} holds for all $k > l \geq 0$. \\
(ii) It would be interesting to determine whether there exist frequencies $\lambda$ that do not satisfy \eqref{spacing}, but for which 
 \eqref{stincl2} holds for all $k > l \geq 0$; we believe that the work of Meyer \cite{Meyer} on Kahane's property $Q(\Lambda)$ \cite{Kahane} might offer inspiration for this problem.
 \end{remark}
\section{Bohr's theorem of order $k$}\label{sect-5}
In this main section, we present a new condition on frequencies implying Bohr's theorem of order $k$.

\subsection{The condition (SNC) and statement of the main result}\label{subs} A frequency $\lambda$ is said to satisfy condition (SNC) if
$$
\forall \delta>0 \, \exists C>0 \, \forall n \in \N \, \exists m> n: \log\left(\frac{\lambda_{m}}{\lambda_{m}-\lambda_n}\right) +\sum_{j=n+1}^{m-1} \frac{\lambda_j-\lambda_n}{\varepsilon_j} \leq C e^{\delta \lambda_n},
$$
We have that 
\begin{equation}
\label{freqcond}
\mbox{(LC) $\Rightarrow$ (SNC) $\Rightarrow$ (NC)},
\end{equation}
the first implication follows from \cite[Lemma 2.1]{bayart}, while the second one is clear.
Similarly to (NC), the condition (SNC) is a separation condition allowing for non-consecutive gaps $\lambda_m - \lambda_n$, $m > n$, unlike (LC). However, in contrast to (NC), it also requires that the intermediate gaps $\lambda_{j+1} - \lambda_j$, $j = n, \ldots, m-1$, do not become too small relative to each other. The following two examples illustrate this and show that the implications in \eqref{freqcond} are strict.

\begin{example}
Let $(\mu_j)_{j \in \N}$ be a frequency satisfying (LC) and $N \in \N$. Let $\varepsilon = (\varepsilon_j)_{j \in \N}$ be a sequence of positive numbers such that $\varepsilon_j \leq \min\{ 1,(\mu_{j+1} - \mu_{j})/(2N)\}$ for all $j \in \N$. Define the frequency  $\lambda = (\lambda_n)_{n \in \N}$ by $\lambda_n = \mu_{j+1} + l \varepsilon_j$ for $n = (N+1)j + l$, with  $j \in \N_0$, $l \in \{0, \ldots,N\}$. Then, $\lambda$ satisfies (SNC) for every choice of $\varepsilon$, but it satisfies (LC) if and only if $e^{-\delta \mu_j}= O(\varepsilon_j)$ for every $\delta >0$.
\end{example}
 
\begin{example}
Let $(\mu_j)_{j \in \N}$ be a frequency satisfying (LC) and $\varepsilon_l = (\varepsilon_{l,j})_{j \in \N}$, $l = 1,2$, be two sequences of positive numbers such that $\varepsilon_{1,j} + \varepsilon_{2,j} \leq (\mu_{j+1} - \mu_{j})/2$ for all $j \in \N$. Define the frequency  $\lambda = (\lambda_n)_{n \in \N}$ by 
\[
\lambda_n =
\begin{cases}
\mu_{j+1}, & n = 3j,     \quad j \in \mathbb{N}_0, \\
\mu_{j+1} + \varepsilon_{1,j}, & n = 3j + 1, \quad j \in \mathbb{N}_0, \\
\mu_{j+1} + \varepsilon_{1,j} + \varepsilon_{2,j}, & n = 3j + 2, \quad j \in \mathbb{N}_0.
\end{cases}
\]
 Then, $\lambda$ satisfies (NC) for every choice of $\varepsilon_1$ and $\varepsilon_2$, but  it does not satisfy (SNC) if $\varepsilon_{1,j} = O(e^{-\delta \mu_j})$ and $\varepsilon_{2,j} = O(\varepsilon_{1,j}e^{-\delta \mu_j})$ for some  $\delta >0$. 
\end{example}
Let  $k \geq 0$. Recall from the introduction that a frequency $\lambda$ is said to satisfy \emph{Bohr's theorem of order $k$} if $\sigma^{\lambda,k}(D) \leq 0$ for every $D \in \H k$. We refer to \cite[Theorems 2.8–2.10]{bohrpoly} for various equivalent formulations of Bohr's theorem of order $k$; in particular, we adopt here the characterization given in \cite[Theorem 2.9]{bohrpoly} as definition.  The following theorem is the main result of this article.

\begin{theorem}\label{main-article}
Let $\lambda$ be a frequency satisfying (SNC) and $k \geq 0$. Then, $\lambda$ satisfies Bohr's theorem of order $k$.
\end{theorem}

The remainder of this section is devoted to showing Theorem \ref{main-article}. We now give an outline of the proof. The main step consists in establishing a suitable maximal inequality for the norms $\|R^{\lambda,k}_x\|_{\mathcal{L}(\H k)}$, from which the result will follow by the Bohr–Cahen formula together with a technical result about the condition (SNC).

If $k \in \N_0$, we prove this maximal inequality by adapting Bayart's method from \cite[Theorem 3.2]{bayart}. For $D \in \H k$, $n,m \in \N$, $m>n$, and $x \in [\lambda_n, \lambda_{n+1}]$, we write
$$
R^{\lambda,k}_x(D)(s) = \sum_{j=1}^\infty a_j(D) \varphi(\lambda_j)e^{-\lambda_j s}-\sum_{j=n+1}^{m-1}a_j(D)\varphi(\lambda_{j})e^{-\lambda_{j}s}, \qquad s \in \C_0,
$$
where $\varphi \in \mathcal{D}(\mathbb{R})$ is a function with $\varphi(y)=\left(1-\frac{y}{x}\right)^K$ for $y \in [0,\lambda_n]$ and $\operatorname{supp}\varphi \subseteq (-\infty,\lambda_m]$ such that it has good Fourier decay and satisfies suitable bounds on $[\lambda_n,\lambda_m]$. The first sum can then be  estimated using the Saksman vertical convolution formula, while the second one can be treated via the  coefficient bounds established in Theorem \ref{coeff_bounds}. The principal difficulty is the construction of a suitable $\varphi$.

Finally, we reduce the general (non-integral) case to the integral one and Proposition \ref{lkasymp} by a technique going back to Hardy and Riesz \cite[Theorem 22]{hardy}; see also \cite[Theorem 3.4]{bohrpoly}.

\begin{remark}
We do not know whether Bohr's theorem of order $k$ holds under the condition (NC).
\end{remark}

\subsection{A compactly supported extension of the Riesz kernel with good Fourier decay}
Throughout this subsection, we fix $1 \leq \alpha < \beta$ with $\beta -\alpha \leq 2$ and  $x \in [\alpha,\beta]$.  Let $K \in \N_0$. Our goal is to construct a function $\varphi \in \mathcal{D}(\R)$  such that $\varphi(y) = \left(1-\frac{y}{x}\right)^K$  for $y \in [0, \alpha]$,  $\supp \varphi \subseteq (-\infty, \beta]$, and for which the integrals
$$
\int_{-\infty}^\infty | \widehat{\varphi}(t)| (1+\abs{t})^k {\rm d} t
$$
admit suitable bounds for $k \geq K$. Moreover, we will  establish estimates for $\varphi$ on $[\alpha, \beta] $. Our construction of $\varphi$ is inspired by the trapezoid function used by Bayart in \cite[Theorem 3.2]{bayart} (corresponding to the case $K = 0$), but is much more involved.

We will use repeated convolutions of indicator functions, sometimes referred to as cardinal $B$-splines. Given $a >0$, we write $\chi_a$ for the indicator function of $[-a,a]$. For $n \in \N$, we define 
$$
\chi_{n,a} = \underbrace{\chi_a \ast \cdots \ast \chi_a}_{n \text{ times}}.
$$
We will need the following properties of these functions.
\begin{lemma}
\label{xi}
Let $a>0$ and $n \in \N$.
\begin{itemize}
\item[(i)] $\chi_{n,a}$ is piecewise polynomial and, for $n \geq 2$, belongs to $C^{n-2}(\R)$.
\item[(ii)] $\supp \chi_{n,a} \subseteq [-na,na]$.
\item[(iii)]  We have that
\[\chi_{n,a}(y)=\frac{1}{(n-1)!}\left(na-y\right)^{n-1}, \qquad  y\in[(n-2)a, na],\]
and
    \[\chi_{n,a}(y)=\frac{1}{(n-1)!}\left(na+y\right)^{n-1}, \qquad y\in[-na, -(n-2)a].\]
\item[(iv)] If $n \geq 2$, then $\| \chi_{n,a}^{(n-2)}\|_\infty \leq 2^{n-1} a$.
\item[(v)] We have that
$$
\widehat{\chi}_{n,a}(t) = 2^{n} \left( \frac{\sin(at)}{t} \right)^{n}, \qquad t \in \R.
$$
\end{itemize}
\end{lemma}
\begin{proof}
The properties (i)--(iv) can be shown by induction on $n$; see also \cite[Sections 4.1 and 4.2]{Chui}. For (v), note that
$$
\widehat{\chi}_{n,a}(t) = \left(\widehat{\chi}_{a}(t)\right)^{n} = 2^{n}\left( \frac{\sin(at)}{t} \right)^{n}, \qquad t \in \R.
$$
\end{proof}

Our construction of $\varphi$ consists of two steps. First, we employ  Lemma \ref{xi} to find a suitable piecewise polynomial extension of the Riesz kernel $\left(1-\frac{\cdot}{x}\right)^K$. This is done in the next lemma. 
\begin{lemma}\label{cstr-part1}
Let $K \in \N$ and set $b = \frac{\beta - \alpha}{K+2}$. There exists a  compactly supported, piecewise polynomial function $\psi  \in C^{K-1}(\R)$ satisfying the following properties:
\begin{itemize}
\item[(i)]  $\psi(y)=\left(1-\frac{y}{x}\right)^K$ for $y\in[-b, \alpha+b]$.
\item[(ii)] $\supp \psi \subseteq (-\infty,\beta-b]$.
\item[(iii)]   Set $\kappa(y) = \psi(y) - \left(1-\frac{y}{x}\right)^K$. Then, 
$$
|\kappa^{(K-1)}(y)| \leq K!(2^{K-1} + K+3)b, \qquad y\in[\alpha+b, \beta +b].
$$
\item[(iv)] We have that
$$
|\widehat{\psi}(t)| \leq \frac{2^{K+1} K!}{x^K} \left( \left | \frac{\sin(a_1t)}{t} \right|^{K+1} + \left | \frac{\sin(a_2t)}{t} \right|^{K+1} \right), \qquad t \in \R,
$$
for $a_1 = (x+b)/2$ and some $a_2 \in [0, b/2]$.
\end{itemize}
\end{lemma}
\begin{proof}
Define
    \[\psi_1(y)=\frac{K!}{x^K}\chi_{K+1,a_1}(y-x+(K+1)a_1), \quad \mbox{ with }a_1=\frac{x+b}{2}.\]
  From Lemma \ref{xi}(i)--(iii) we obtain that $\psi_1 \in C^{K-1}(\R)$ is a compactly supported, piecewise polynomial function such that $\psi_1(y)=\left(1-\frac{y}{x}\right)^K$ for $y\in[-b, x]$ and $\operatorname{supp} \psi_1 \subseteq (-\infty, x]$.
    
    In the remainder of the proof, we distinguish three cases: $x<\alpha+b$  (I); $\alpha+b \leq x \leq \beta-b$   (II); $x>\beta-b$ (III). 
    
    We start with case (II), which is the easiest. Set $\psi = \psi_1$. It is clear that $\psi$ satisfies (i) and (ii), while (iv) follows from Lemma \ref{xi}(v).  Next, we prove (iii). Note that $\kappa(y) = 0$ for $y \in [\alpha + b, x]$. If $y \in [x,\beta +b]$, then $\kappa(y) = - \left(1-\frac{y}{x}\right)^K$ and thus
    $$
 |\kappa^{(K-1)}(y)| = \frac{K!}{x^K} (y-x) \leq K!(\beta - \alpha + b)    = K!(K+3)b.
    $$
    This shows (iii).
    
    Next, we consider case (I). The function $\psi_1$ now becomes identically zero from $x$ onwards, which is too soon. We fix this by adding a  function to $\psi_1$. Define 
    \[\psi_2(y)=\frac{(-1)^KK!}{x^K}\chi_{K+1,a_2}(y-x-(K+1)a_2), \quad \mbox{ with }a_2=\frac{\alpha+b-x}{2}.\]
  By Lemma \ref{xi}(i)--(iii),  we get that $\psi_2 \in C^{K-1}(\R)$  is a compactly supported, piecewise polynomial function such that $\psi_2(y)=\left(1-\frac{y}{x}\right)^K$ for $y\in[x, \alpha + b]$ and $\operatorname{supp} \psi_2 \subseteq [x,x+(K+1)(\alpha+b-x)]$. Set $\psi = \psi_1 +\psi_2$. Then, (i) is clear, (ii) follows from the inequality 
  $$
  x+(K+1)(\alpha+b-x) \leq \beta - b,
  $$
and  (iv) is a consequence of Lemma \ref{xi}(v).
 Next, we prove (iii).  For $y \in [\alpha + b, \beta+b]$, it holds that $\kappa(y) =  \psi_2(y) - \left(1-\frac{y}{x}\right)^K$  and thus, by Lemma \ref{xi}(iv),
    $$
 |\kappa^{(K-1)}(y)| \leq \frac{K!}{x^K} ( 2^Ka_2 + (y-x)) \leq K!(2^{K-1} + K+3)b.
    $$

Finally, we treat case (III). Now, the support of  $\psi_1$ is too large.
 We fix this by subtracting a  function from $\psi_1$. Define 
    \[\psi_2(y)=\frac{K!}{x^K}\chi_{K+1,a_2}(y-x+(K+1)a_2), \quad \mbox{ with }a_2=\frac{x+b-\beta}{2}.\]
 By Lemma \ref{xi}(i)--(iii),  we get that $\psi_2 \in C^{K-1}(\R)$  is a compactly supported, piecewise polynomial function such that $\psi_2(y)=\left(1-\frac{y}{x}\right)^K$ for $y\in[\beta-b,x]$ and $\operatorname{supp} \psi_2 \subseteq [x-(K+1)(x+b-\beta),x]$. Set $\psi = \psi_1 -\psi_2$.   Then, (ii) is clear, (i) follows from the inequality 
  $$
 \alpha + b \leq  x - (K+1)(x+b-\beta),
  $$
and  (iv) is a consequence of Lemma \ref{xi}(v). Next, we prove (iii).  It holds that $\kappa(y) =  \psi_2(y)$ for $y \in [\alpha+b, x]$ and $\kappa(y) = - \left(1-\frac{y}{x}\right)^K$ for $y \in [x, \beta+b]$. Frow Lemma \ref{xi}(iv) we obtain that
$$
 |\kappa^{(K-1)}(y)| \leq \frac{K!}{x^K} ( 2^Ka_2 + (y-x)) \leq K!(2^{K-1} + K+3)b.
    $$
 \end{proof}
In the second part of the construction, we improve the Fourier decay of the function $\psi$ from Lemma \ref{cstr-part1} by convolving it with a compactly supported smooth function. To ensure that the convolution still agrees with the Riesz kernel $\left(1-\frac{\cdot}{x}\right)^K$ on $[\alpha,\beta]$, we choose this function to have total mass $1$ and vanishing moments of orders $1$ to $K$.
\begin{proposition}\label{cstr-part2} Let $K \in \N$. There exists  $\varphi \in \mathcal{D}(\R)$ satisfying the following properties:
\begin{itemize}
\item[(i)]  $\varphi(y)=\left(1-\frac{y}{x}\right)^K$ for $y\in[0, \alpha]$.
\item[(ii)] $\supp \varphi\subseteq (-\infty,\beta]$.
\item[(iii)]  There is $C >0$ (independent of $\alpha$, $\beta$, and $x$) such that
$$
|\varphi(y)| \leq C (y-\alpha)^K, \qquad y \in [x, \beta].
$$
\item[(iv)] For every $k \geq K$ there is $C >0$ (independent of $\alpha$, $\beta$, and $x$) such that
$$
\int_{-\infty}^\infty | \widehat{\varphi}(t)| (1+\abs{t})^k {\rm d} t \leq  \frac{C}{(\beta -\alpha)^{k-K}} \left(1+\log\left(\frac{\beta}{\beta-\alpha}\right) \right).
$$
\end{itemize}
\end{proposition}
\begin{proof}
Set  $b = \frac{\beta - \alpha}{K+2}$. Choose $\theta \in \mathcal{D}((-1,1))$ such that 
\begin{equation}
\label{vanmom}
\int_{-\infty}^\infty \theta(y){ \rm d} y = 1, \qquad \int_{-\infty}^\infty y^j\theta(y){ \rm d} y  = 0, \qquad j = 1, \ldots, K,
\end{equation}
and write $\theta_b(y)= b^{-1}\theta(y/b)$. Let $\psi$ be as in Lemma \ref{cstr-part1} if $K \geq 1$ and $\psi = \chi_a$ with $a  = \frac{\alpha + \beta}{2}$ if $K = 0$. Define
$\varphi =\psi \ast \theta_b$. We now show that $\varphi$ satisfies all requirements.

Properties (i)--(iii) are obvious if $K = 0$. Assume that $K \geq 1$. Then, (i) follows from  Lemma \ref{cstr-part1}(i) and  \eqref{vanmom}, while (ii) is a consequence of Lemma \ref{cstr-part1}(ii). We now  show (iii). As in  Lemma \ref{cstr-part1}(iii),  we define $\kappa(y) = \psi(y) - \left(1-\frac{y}{x}\right)^K$. By \eqref{vanmom}, we obtain that, for $y \in [x, \beta]$,
$$
\varphi(y) = \kappa \ast \theta_b(y) + \left(1-\frac{\cdot}{x}\right)^K*\theta_b(y) = \kappa \ast \theta_b(y) + \left(1-\frac{y}{x}\right)^K.
$$
Note that
$$
 \left|1-\frac{y}{x}\right|^K = \frac{(y-x)^K}{x^K} \leq (y-\alpha)^K.
$$
We now bound the first term $\kappa \ast \theta_b(y)$.  Lemma \ref{cstr-part1}(i) yields that $\kappa^{(j)}(w)=0$ for  $w \in [-b, \alpha + b]$ and $j=0,\dots, K-1$. Hence, by Taylor's theorem, we find that
   \begin{align*}
\kappa \ast \theta_b(y) &= \int_{y-b}^{y+b} \kappa(w)\theta_b(y-w)\text{d}w = \int_{\max\{y-b, \alpha+b\}}^{y+b} \kappa(w)\theta_b(y-w)\text{d}w  \\
& =   \frac{1}{(K-1)!} \int_{\max\{y-b, \alpha+b\}}^{y+b} \kappa^{(K-1)}(\xi)(w-(\alpha+b))^{K-1}\theta_b(y-w)\text{d}w,
\end{align*}
for some $\xi = \xi_w \in [\alpha + b, w]$. Lemma  \ref{cstr-part1}(iii) implies that $|\kappa^{(K-1)}(\xi)| \leq K!(2^{K-1} + K+3)b$.
We get that
   \begin{align*}
|\kappa \ast \varphi_b(y)| &\leq  \frac{1}{(K-1)!} \int_{{\max\{y-b, \alpha+b\}}}^{y+b} |\kappa^{(K-1)}(\xi)| (w-(\alpha +b))^{K-1} |\theta_b(y-w)|\text{d}w \\
& \leq K(2^{K-1} + K+3)b \int_{{\max\{y-b, \alpha+b\}}}^{y+b} (w-(\alpha +b))^{K-1} |\theta_b(y-w)|\text{d}w \\
&\leq (2^{K-1} + K+3) \| \theta \|_\infty (y-\alpha)^K.
     \end{align*}
This proves (iii).

Finally, we consider (iv).  Fix $k \geq K$. We start by showing that there is $C >0$ such that, for all $a>0$ and $0 < \varepsilon \leq 1$,
\begin{equation}
\label{eqinproof}
\int_{-\infty}^\infty  \left |\frac{\sin(at)}{t}\right|^{K+1}|\widehat{\theta}(\varepsilon t)|(1+\abs{t})^k\dt \leq C \frac{(1+a^K)}{\varepsilon^{k-K}} \left(1+\log_+\left(\frac{a}{\varepsilon}\right) \right).
\end{equation}
We have that
\begin{align*}
&\int_{-\infty}^\infty  \left |\frac{\sin(at)}{t}\right |^{K+1}|\widehat{\theta}(\varepsilon t)|(1+\abs{t})^k\dt \\
&= \frac{1}{\varepsilon^{k-K}}\left(\int_{|t| \leq 1} + \int_{|t| \geq 1} \right)\left |\frac{\sin\left(\frac{a}{\varepsilon}t\right)}{t}\right|^{K+1}|\widehat{\theta}(t)|\left(\varepsilon+|t|\right)^k\dt.
\end{align*}
For the first integral, we note that  $(\varepsilon+|t|)^k \leq 2^{k-K}(\varepsilon+|t|)^K$ if  $t \in [-1,1]$. Hence,
\begin{align*}
&\int_{|t| \leq 1}\left |\frac{\sin\left (\frac{a}{\varepsilon}t\right)}{t}\right|^{K+1}|\widehat{\theta}(t)|\left(\varepsilon+|t|\right)^k\dt \\
&\leq 2^{k-K+1} \| \widehat{\theta}\|_\infty \int_{0}^1 \left |\frac{\sin\left(\frac{a}{\varepsilon}t\right)}{t}\right|^{K+1}\left(\varepsilon+t\right)^K\dt \\
&= 2^{k-K+1} \| \widehat{\theta}\|_\infty \int_{0}^{a/\varepsilon} \left |\frac{\sin t}{t}\right |^{K+1}  \left(a+t\right)^K\dt \\
&\leq 2^{k+1} \| \widehat{\theta}\|_\infty \int_{0}^{a/\varepsilon}\left | \frac{\sin t}{t}\right |^{K+1}\left(a^K+t^K\right)\dt \\
&\leq 2^{k+1} \| \widehat{\theta}\|_\infty (1+a^K) \int_{0}^{a/\varepsilon}\left |\frac{\sin t}{t}\right |\dt \\
&\leq 2^{k+1} \| \widehat{\theta}\|_\infty (1+a^K)\left(1+\log_+\left(\frac{a}{\varepsilon}\right) \right).
 \end{align*}
 For the second integral, we find that
$$
 \int_{|t| \geq 1}\left |\frac{\sin\left(\frac{a}{\varepsilon}t\right)}{t} \right |^{K+1}|\widehat{\theta}(t)|\left(\varepsilon+|t|\right)^k\dt \leq \int_{|t| \leq 1} |\widehat{\theta}(t)|\left(1+|t|\right)^k\dt.
$$
This shows \eqref{eqinproof}.

For $K =0$, the result now follows from Lemma \ref{xi}(v) and  \eqref{eqinproof}. Assume that $K \geq 1$.
 As in Lemma \ref{cstr-part1}(iv), we set $a_1 = (x+b)/2$. Then,  Lemma \ref{cstr-part1}(iv),  \eqref{eqinproof}, and $a_1 \geq b/2$   imply that
\begin{align}
\label{eqinproof2}
&\int_{-\infty}^\infty | \widehat{\varphi}(t)| (1+\abs{t})^k\dt = \int_{-\infty}^\infty  |\widehat{\psi}(t)| |\widehat{\theta}(bt)| (1+\abs{t})^k\dt  \\ \nonumber
& \leq \frac{2^{K+2}K!C}{x^K}\frac{(1+a_1^K)}{b^{k-K}} \left(1+\log_+\left(\frac{a_1}{b}\right) \right),
\end{align}
Since $b \leq 1 \leq x \leq \beta$, we obtain that $1 + a_1^K \leq 2x^K$ 
and 
$$
\frac{a_1}{b} = \frac{K+2}{2} \frac{x+b}{\beta-\alpha} \leq (K+2) \frac{\beta}{\beta-\alpha}.
$$
The result now follows by plugging these inequalities into \eqref{eqinproof2}.
\end{proof}

\subsection{Maximal inequalities} We are ready to show the maximal inequalities underlying Theorem \ref{main-article}. We first consider Riesz means of integral order.

\begin{proposition}\label{max-ineq-1}
Let  $\lambda$ be a frequency, $K \in \N_0$, and $k \geq K$. There exists $C >0$ such that, for all $n,m \in \N$, $m >n$, with $\lambda_m  - \lambda_n \leq 2$, 
    \[ \sup_{\lambda_n \leq x \leq \lambda_{n+1}}\norm{R_x^{\lambda,K}}_{\mathcal{L}(\H k)}\leq \frac{C}{(\lambda_{n+1} - \lambda_n)^{k-K}}\left(\log\left(\frac{\lambda_m}{\lambda_m-\lambda_n}\right)+\sum_{j=n+1}^{m-1}\left(\frac{\lambda_j-\lambda_n}{\varepsilon_j}\right)^k\right).\]
\end{proposition}
\begin{proof}
Fix $n,m \in \N$, $m >n$, with $\lambda_m  - \lambda_n \leq 2$. By Theorem \ref{coeff_bounds}, we may assume that $\lambda_m \geq 2e$. Let $x \in [\lambda_n, \lambda_{n+1}]$. Choose $\varphi$ as in Proposition \ref{cstr-part2} with $\alpha = \lambda_n$ and $\beta = \lambda_m$. Then, for all $D \in \H k$, 
\[R_x^{\lambda, K}(D)(s)=\sum_{j=1}^\infty a_j(D) \varphi(\lambda_j)e^{-\lambda_j s}-\sum_{j=n+1}^{m-1}a_j(D)\varphi(\lambda_{j})e^{-\lambda_{j}s}, \qquad s \in \C_0.\]
Propositions \ref{saksmanth} and \ref{cstr-part2}(iv) imply that there is $C_1 >0$ (independent of $n$, $m$, and $x$) such that
\begin{align*}
&\left |\sum_{j=1}^\infty a_j(D) \varphi(\lambda_j)e^{-\lambda_j s} \right| = \frac{1}{2\pi} \left | \int_{-\infty}^\infty f_D(s+it) \widehat{\varphi}(-t) {\rm d} t \right  | \\&\leq \frac{1}{2\pi} \| D\|_{\H k}(1+|s|)^k \int_{-\infty}^\infty | \widehat{\varphi}(t)| (1+\abs{t})^k {\rm d} t  \\
& \leq   \frac{C_1}{2\pi(\lambda_m -\lambda_n)^{k-K}}\left(1+\log\left(\frac{\lambda_m}{\lambda_m-\lambda_n}\right) \right) \| D\|_{\H k}(1+|s|)^k  \\ 
& \leq  \frac{C_1}{\pi(\lambda_{n+1} -\lambda_n)^{k-K}} \log\left(\frac{\lambda_m}{\lambda_m-\lambda_n}\right) \| D\|_{\H k}(1+|s|)^k . 
\end{align*}
By Theorem \ref{coeff_bounds} and Proposition  \ref{cstr-part2}(iv), there is $C_2 >0$ (independent of $n$, $m$, and $x$) such that
\begin{align*}
\left |\sum_{j=n+1}^{m-1}a_j(D)\varphi(\lambda_{j})e^{-\lambda_{j}s} \right| &\leq C_2 \| D\|_{\H k} \sum_{j=n+1}^{m-1}\left(\frac{1}{\varepsilon_j}\right)^k (\lambda_j-\lambda_n)^K \\
&\leq   \frac{C_2}{\pi(\lambda_{n+1} -\lambda_n)^{k-K}}  \| D\|_{\H k} \sum_{j=n+1}^{m-1}\left(\frac{\lambda_j-\lambda_n}{\varepsilon_j}\right)^k.
\end{align*}
This shows the result.
\end{proof}

Next, we present a maximal inequality for general (non-integral) orders. Given a function $m: \N \to \N$, we define  $\tilde{m}(n)=\sup\{ j \in \N \mid m(j)\leq n\}$.
\begin{proposition}\label{max-ineq-2}
Let $\lambda$ be a frequency and $k \geq 0$. Let $m: \N \to \N$ be a function such that $m(n) > n$ and $\lambda_{m(n)}  - \lambda_n \leq 2$ for all $n \in \N$. There exists $C >0$ such that, for all $n \in \N$,
    \[  \sup_{\lambda_n \leq x \leq \lambda_{n+1}}\norm{R_x^{\lambda,k}}_{\mathcal{L}(\H k)}\leq C \sum_{l=\tilde{m}(n)}^{n}\left(\log\left(\frac{\lambda_{m(l)}}{\lambda_{m(l)}-\lambda_l}\right) +\sum_{j=l+1}^{m(l)-1}\left(\frac{\lambda_j-\lambda_l}{\varepsilon_j}\right)^k\right).\]
\end{proposition}
    \begin{proof}
If $k \in \N$, the result follows from  Proposition \ref{max-ineq-1} with $K = k$. Assume that $k \notin \N$ and set $K=\floor{k}$. Fix $n \in \N$ and set $N = \tilde{m}(n)$. By Theorem \ref{coeff_bounds}, it suffices to show  the result for $N$ large. Let $x \in [\lambda_n, \lambda_{n+1}]$.  Lemma \ref{technical-0} yields that, for all $s \in \C_0$,
        \begin{align}
          \label{splitint}
          &  S_x^{\lambda,k}(D)(s) = \frac{\Gamma(k +1)}{K!\Gamma(k-K)}\int_0^x S_y^{\lambda, K} (D)(s)(x-y)^{k-K-1}{\rm d} y\\ \nonumber
            &= \frac{\Gamma(k +1)}{K!\Gamma(k-K)} \left (\int_0^{\lambda_{N}}  + \sum_{l=N}^{n-1}\int_{\lambda_{l}}^{\lambda_{l+1}} + \int_{\lambda_{n}}^{x}   \right) S_y^{\lambda, K} (D)(s)(x-y)^{k-K-1}{\rm d} y.        \end{align}
We start with the first integral. Let $\varepsilon >0$ be any number such that $k-K+\varepsilon<1$ (the value of $\varepsilon$ will be chosen later). By the second mean value theorem for integrals, there is $\xi \in [0,\lambda_N]$ such that    
  \begin{align*}
&\int_0^{\lambda_{N}}  S_y^{\lambda, K} (D)(s)(x-y)^{k-K-1}{\rm d} y \\ 
&= \int_0^{\lambda_{N}}  \frac{S_y^{\lambda, K} (D)(s)(x-y)^{k-K  + \varepsilon-1}}{(x-y)^\varepsilon}{\rm d} y  \\
&= \frac{1}{(x-\lambda_{N})^{\varepsilon}} \int_{\xi}^{\lambda_{N}}  S_y^{\lambda, K} (D)(s)(x-y)^{k-K  + \varepsilon-1} {\rm d} y \\
&= \frac{1}{(x-\lambda_{N})^{\varepsilon}}   \left ( \int_{0}^{\lambda_{N}} - \int_{0}^{\xi}   \right) S_y^{\lambda, K} (D)(s)(x-y)^{k-K  + \varepsilon-1} {\rm d} y.
\end{align*}
By applying Lemma \ref{technical} to both of these integrals and using Theorem \ref{lkasymp}, we find that there are $C_1, C_2 >0$ (only depending on $k$) such that
  \begin{align*}
\left|  \int_0^{\lambda_{N}}  S_y^{\lambda, K} (D)(s)(x-y)^{k-K-1}{\rm d} y \right | &\leq C_1\frac{1}{(x-\lambda_{N})^{\varepsilon}}  \sup_{y\leq \lambda_N} |S_y^{\lambda, k+\varepsilon} (D)(s)| \\
&\leq C_1\left(\frac{\lambda_N}{\lambda_{m(N)}-\lambda_{N}}\right)^{\varepsilon}  x^{k} \sup_{y\leq \lambda_N} |R_y^{\lambda, k+\varepsilon} (D)(s)| \\
&\leq \frac{C_2}{\varepsilon}\left(\frac{\lambda_N}{\lambda_{m(N)}-\lambda_{N}}\right)^{\varepsilon} x^k\norm{D}_{\H k}(1+|s|)^k.
  \end{align*}
 We now choose $\varepsilon=1/\log\left(\frac{\lambda_N}{\lambda_{m(N)}-\lambda_N}\right)$, as $\lambda_{m(N)}-\lambda_N \leq 2$, this becomes small enough for large enough $N$. Then, 
\begin{align*}
 &\left|  \int_0^{\lambda_{N}}  S_y^{\lambda, K} (D)(s)(x-y)^{k-K-1}{\rm d} y \right | \\ 
 & \leq eC_2\log\left(\frac{\lambda_N}{\lambda_{m(N)}-\lambda_N}\right)x^k\norm{D}_{\H k}(1+|s|)^k.
\end{align*}
For the other integrals in \eqref{splitint},  Proposition \ref{max-ineq-1} implies that there is $C_3 >0$ (only depending on $k$) such that,                
 for $l = N, \ldots, n-1$, 
\begin{align*}
      &\left |     \int_{\lambda_{l}}^{\lambda_{l+1}}  S_y^{\lambda, K} (D)(s)(x-y)^{k-K-1}{\rm d} y \right |  \\ 
      &\leq \sup_{\lambda_l\leq y \leq \lambda_{l+1}}| S_y^{\lambda, K} (D)(s)|\int_{\lambda_l}^{\lambda_{l+1}}(x-y)^{k-K-1}\text{d}y\\ 
            &\leq  \frac{1}{k-K} x^k  \sup_{\lambda_{l}\leq y\leq \lambda_{l+1}}| R_y^{\lambda, K} (D)(s)| (\lambda_{l+1}-\lambda_{l})^{k-K} \\ 
                 & \leq C_3 \left(\log\left(\frac{\lambda_{m(l)}}{\lambda_{m(l)}-\lambda_l}\right)+\sum_{j=l+1}^{m(l)-1}\left(\frac{\lambda_j-\lambda_l}{\varepsilon_j}\right)^k\right)  x^k \|D\|_{\H k}(1+|s|)^k.
        \end{align*}
        and
\begin{align*}
   &  \left |     \int_{\lambda_{n}}^{x}  S_y^{\lambda, K} (D)(s)(x-y)^{k-K-1}{\rm d} y \right | \\ 
   &\leq \sup_{\lambda_n\leq y \leq x}| S_y^{\lambda, K} (D)(s)|\int_{\lambda_n}^{x}(x-y)^{k-K-1}\text{d}y\\ 
                        &\leq \frac{1}{k-K} x^k  \sup_{\lambda_n\leq y \leq \lambda_{n+1}}| R_y^{\lambda, K} (D)(s)|(\lambda_{n+1}-\lambda_n)^{k-K}. \\
                      &  \leq C_3 \left(\log\left(\frac{\lambda_{m(n)}}{\lambda_{m(n)}-\lambda_n}\right)+\sum_{j=n+1}^{m(n)-1}\left(\frac{\lambda_j-\lambda_n}{\varepsilon_j}\right)^k\right) x^k \|D\|_{\H k}(1+|s|)^k.
                                          \end{align*}  
This shows the result.                            
    \end{proof}
    
\subsection{ Proof of Theorem \ref{main-article}} In order to deduce Theorem \ref{main-article} from the maximal inequality in Proposition \ref{max-ineq-2} together with the Bohr--Cahen formula, we need a technical result about the condition (SNC).  We start with the following lemma.
\begin{lemma} \label{techfreq0}
Let $\lambda$ be a frequency satisfying (SNC) and $k \geq 0$.  For every $\delta >0$, there are $C >0$ and a function $m: \N \to \N$ such that, for all $n \in \N$, $m(n) > n$, $n - \tilde{m}(n) +1 \leq C e^{\delta \lambda_n}$, and
\begin{equation}
\label{step0tech}
\sum_{l=\tilde{m}(n)}^{n}\left(\log\left(\frac{\lambda_{m(l)}}{\lambda_{m(l)}-\lambda_l}\right) +\sum_{j=l+1}^{m(l)-1}\left(\frac{\lambda_j-\lambda_l}{\varepsilon_j}\right)^k\right) \leq Ce^{\delta \lambda_n}.
\end{equation}
\end{lemma}
\begin{proof}
We first observe that, for every $\delta >0$, there are $C >0$ and a function $m: \N \to \N$ such that, for all $n \in \N$, $m(n) > n$ and 
\begin{equation}
\label{step1tech}
\log\left(\frac{\lambda_{m(n)}}{\lambda_{m(n)}-\lambda_n}\right) +\sum_{j=n+1}^{m(n)-1}\left(\frac{\lambda_j-\lambda_n}{\varepsilon_n}\right)^k \leq Ce^{\delta \lambda_n}.
\end{equation}
If $k \leq 1$, this is clear. For $k \geq1$, it follows from the fact that, for all $N \in \N$ and $a_1, \ldots, a_N \geq 0$,
$$
\sum_{j=1}^N a_j^k  \leq \left(\sum_{j=1}^N a_j\right)^k.  
$$
Now let $\delta >0$ be arbitrary. Choose  $C >0$ and a function $m: \N \to \N$ such that, for all $n \in \N$, $m(n) > n$ and \eqref{step1tech} holds. We find that
$$
 m(n)-n-1 \leq \sum_{j=n+1}^{m(n)-1}\left(\frac{\lambda_j-\lambda_n}{\varepsilon_j}\right)^k \leq Ce^{\delta \lambda_n}.
$$
Since $n  < m(\tilde{m}(n) +1)$ and $\tilde{m}(n) +1 \leq n$, we obtain that
$$
n-\tilde{m}(n) +1 <  m(\tilde{m}(n) +1) - (\tilde{m}(n) +1) + 2 \leq Ce^{\delta \lambda_{\tilde{m}(n) +1}} +3 \leq (C+3)e^{\delta \lambda_n}.
$$
This implies that 
$$
\sum_{l=\tilde{m}(n)}^{n}\left(\log\left(\frac{\lambda_{m(l)}}{\lambda_{m(l)}-\lambda_l}\right) +\sum_{j=l+1}^{m(l)-1}\left(\frac{\lambda_j-\lambda_l}{\varepsilon_j}\right)^k\right)        \leq C \sum_{l=\tilde{m}(n)}^{n}e^{\delta \lambda_l} \leq C(C+3)e^{2\delta \lambda_n}.
$$
\end{proof}
    
The following lemma is inspired by \cite[Lemma 4.2]{bayart}.
\begin{lemma} \label{techfreq1}
Let $\lambda$ be a frequency satisfying (SNC) and $k \geq 0$. There exists a frequency $\mu$ with $\mu \supseteq \lambda$ such that, for every $\delta >0$, there are $C >0$ and a function $p: \N \to \N$ satisfying, for all $n \in \N$, $p(n) > n$, $\mu_{p(n)}-\mu_n\leq 2$, and 
\begin{equation}
 \label{ineqtech} 
\sum_{l=\tilde{p}(n)}^{n}\left(\log\left(\frac{\mu_{p(l)}}{\mu_{p(l)}-\mu_l}\right) +\sum_{j=l+1}^{p(l)-1}\left(\frac{\mu_j-\mu_l}{\varepsilon_j(\mu)}\right)^k\right) \leq Ce^{\delta \mu_n}.
\end{equation}
\end{lemma}
\begin{proof}
We recursively define the sequence $(p_j)_{j\in\N}$ of natural numbers by $p_1 = 1$ and $p_{j+1} = p_j + l_j+1$, where $l_j \in \N_0$ is the unique natural number such that $l_j <\lambda_{j+1} - \lambda_j \leq l_j+1$. We  define the frequency $\mu$ as follows: $\mu_{p_j} = \lambda_j$ and $\mu_{p_j +l}$, $l=1, \ldots l_j $, in such a way that  $1/2 \leq \mu_{p_j+l+1}-\mu_{p_j+l}\leq 1$ for all $l = 0, \ldots, l_j$. By definition, $\mu \supseteq \lambda$.

 Let $\delta >0$ be arbitrary. As $\lambda$ satisfies (SNC), Lemma \ref{techfreq0} yields that there are $C_1 >0$ and  $m: \N \to \N$ such that, for all $n \in \N$, $m(n) > n$, $n - \tilde{m}(n) +1 \leq C e^{\delta \lambda_n/2}$, and \eqref{step0tech} holds. We define $p: \N \to \N$ by $p(n) = n+1$ if $n \notin \{k_j \mid j \in \N\}$ and
$$
p(k_j) = \min\{ k_{m(j)}, \inf\{q > k_j \mid q \notin \{k_l \mid l \in \N\}\}, \inf\{q > k_j \mid \mu_q - \mu_{k_j} \geq 1\} \}.
$$
It is clear that $p(n) >n$ for all $n \in \N$. The definitions of $\mu$ and $p$ imply that $\mu_{p(n)}-\mu_n\leq 2$ for all $n \in \N$. We now show that there is $C >0$ such that \eqref{ineqtech} holds for all $n \in \N$  (with $2\delta$ instead of $\delta$). Let $C_2 >0$ be such that $\log(1+2x) \leq C_2e^{\delta x}$ for all $x \geq 0$.
 We distinguish between two cases. First assume that $n \notin \{k_j \mid j \in \N\}$. Then, $p(n)=n+1$ and $\tilde{p}(n) = n-1$. Hence,
    \begin{align*}
        \sum_{l=\tilde{p}(n)}^{n}\Bigg(\log\left(\frac{\mu_{p(l)}}{\mu_{p(l)}-\mu_l}\right)&+\sum_{j=l+1}^{p(l)-1}\left(\frac{\mu_j-\mu_l}{\varepsilon_j(\mu)}\right)^k\Bigg)\\
        &=\log\left(\frac{\mu_{n}}{\mu_{n}-\mu_{n-1}}\right)+\log\left(\frac{\mu_{n+1}}{\mu_{n+1}-\mu_n}\right)\\
        &=\log\left(1+\frac{\mu_{n-1}}{\mu_{n}-\mu_{n-1}}\right)+\log\left(1+\frac{\mu_n}{\mu_{n+1}-\mu_n}\right)\\
        &\leq \log\left(1+2\mu_{n-1}\right)+\log\left(1+2\mu_n\right)\\
        &\leq 2C_2e^{\delta \mu_n}.
    \end{align*}
   Next, suppose that $n = k_j$ for some $j \in \N$. We estimate the two sums
\begin{equation}
\label{twosums}
    \sum_{l=\tilde{p}(n)}^{n}\log\left(\frac{\mu_{p(l)}}{\mu_{p(l)}-\mu_l}\right) \qquad \mbox{and} \qquad     \sum_{l=\tilde{p}(n)}^{n}\sum_{j=l+1}^{p(l)-1}\left(\frac{\mu_j-\mu_l}{\varepsilon_j(\mu)}\right)^k
 \end{equation}
   separately. For the second one, by definition  of $\mu$ and $p$, we find that 
$$
        \sum_{l=\tilde{p}(n)}^{n}\sum_{q=l+1}^{p(l)-1}\left(\frac{\mu_q-\mu_l}{\varepsilon_q(\mu)}\right)^k\leq 2\sum_{l=\tilde{m}(j)}^{j}\sum_{q=l+1}^{m(l)-1}\left(\frac{\lambda_q-\lambda_l}{\varepsilon_q(\lambda)}\right)^k \leq 2Ce^{\delta \lambda_j}=2Ce^{\delta \mu_n}.
$$
Next, we consider the first sum in \eqref{twosums}. Define
    \[E=\{l \in [\tilde{p}(n), n] \mid l  = k_q \mbox{ and } p(k_q) = k_{m(q)} \mbox{ for some } q \in \N\}\]
and $\bar{E}  = [\tilde{p}(n), n] \backslash E$. Then,
   \begin{equation}
   \label{6800}\sum_{l=\tilde{p}(n)}^{n}\log\left(\frac{\mu_{p(l)}}{\mu_{p(l)}-\mu_l}\right)= \left ( \sum_{l \in E} + \sum_{l \in \bar{E}}\right) \log\left(\frac{\mu_{p(l)}}{\mu_{p(l)}-\mu_l}\right)  .
   \end{equation}
   For the first sum, we obtain that 
   $$
    \sum_{l \in E} \log\left(\frac{\mu_{p(l)}}{\mu_{p(l)}-\mu_l}\right) \leq \sum_{l=\tilde{m}(j)}^{j} \log\left(\frac{\lambda_{m(l)}}{\lambda_{m(l)}-\lambda_l}\right) \leq Ce^{\delta \lambda_j}=Ce^{\delta \mu_n}.
$$
    For the second sum, note that $\mu_{p(l)}-\mu_l\geq 1/2$ for all $l \in \bar{E}$. Therefore,
$$
         \sum_{l \in \bar E} \log\left(\frac{\mu_{p(l)}}{\mu_{p(l)}-\mu_l}\right) \leq \sum_{l\in\bar{E}} \log(1+2\mu_{l}) \leq C_2(\tilde{p}(n) - n +1) e^{\delta \mu_{n}}
         $$
 The definition of $p$ implies that $\tilde{p}(n) - n \leq \tilde{m}(j) -j$. Since $ \tilde{m}(j) -j +1 \leq Ce^{\delta \lambda_j} = Ce^{\delta \mu_n}$, we obtain that
$$
      \sum_{l \in \bar E} \log\left(\frac{\mu_{p(l)}}{\mu_{p(l)}-\mu_l}\right) \leq C_2(\tilde{p}(n) - n +1) e^{\delta \mu_{n}}  \leq CC_2e^{2\delta \mu_n}.
$$
This shows the result.
\end{proof}
 We are ready to show  Theorem \ref{main-article}.
 \begin{proof}[Proof of Theorem \ref{main-article}] In view of Lemma \ref{techfreq0}, we may assume without loss of generality that,  for every $\delta >0$, there are $C >0$ and a function $m: \N \to \N$ such that, for all $n \in \N$, $m(n) > n$, $\lambda_{m(n)}-\lambda_n\leq 2$, and 
$$ 
\sum_{l=\tilde{m}(n)}^{n}\left(\log\left(\frac{\lambda_{m(l)}}{\lambda_{m(l)}-\lambda_l}\right) +\sum_{j=l+1}^{m(l)-1}\left(\frac{\lambda_j-\lambda_l}{\varepsilon_j}\right)^k\right) \leq Ce^{\delta \lambda_n}.
$$
Hence, Theorem \ref{max-ineq-1} implies that, for every $\delta >0$, there is $C >0$ such that, for all $n \in \N$,
$$
\sup_{\lambda_n \leq x \leq \lambda_{n+1}}\norm{R_x^{\lambda,k}}_{\mathcal{L}(\H k)}\leq Ce^{\delta \lambda_n}.
$$
Consequently, for all $x >0$ and $D \in \H k$,
$$
 |R_x^{\lambda,k}(D)(0)| \leq \norm{R_x^{\lambda,k}(D)}_{\H k}\leq Ce^{\delta x} \| D\|_{\H k}.
 $$
The result now follows from the Bohr–Cahen formula (Proposition \ref{Bohr-Cahen}).
\end{proof}

\noindent 
\textbf{Acknowledgement.} We would like to thank Frederik Broucke for helpful discussions on the topic of this paper.



 


\begin{thebibliography}{99}

\bibitem{BCQ} R.~Balasubramanian, B.~Calado, H.~Queff{\'e}lec, \emph{The Bohr inequality for ordinary Dirichlet series}, Studia Math. \textbf{175} (2006), 285--304. 

\bibitem{bayart}
F.~Bayart,
\emph{Convergence and almost sure properties in Hardy spaces of Dirichlet series},
Math. Ann. \textbf{382} (2022), 1485--1515.

\bibitem{first_saksman}
F.~Bayart, H.~Queff{\'e}lec, K.~Seip,
\emph{Approximation numbers of composition operators on $H^p$ spaces of Dirichlet series},
Ann. Inst. Fourier (Grenoble) \textbf{66} (2016), 551--588.

\bibitem{Bohr}
H.~Bohr,
\emph{\"Uber die gleichm\"a\ss ige Konvergenz Dirichletscher Reihen},
J. Reine Angew. Math. \textbf{143} (1913), 203--211.

\bibitem{Chui} C.~K.~Chui, \textit{An introduction to wavelets}, Academic Press, 1992. 

\bibitem{DS} G.~Debruyne, D.~Seifert, \emph{An abstract approach to optimal decay of functions and operator semigroups}, Isr. J. Math. \textbf{233} (2019), 439--451.

\bibitem{finorder}
A.~Defant, I.~Schoolmann,
\emph{Holomorphic functions of finite order generated by Dirichlet series},
Banach J. Math. Anal. \textbf{16} (2022), 33

\bibitem{bohrpoly}
A.~Defant, I.~Schoolmann,
\emph{Scaling a theorem of Harald Bohr},
J. Math. Anal. Appl. \textbf{529} (2024), 127324.

\bibitem{hardy}
G.~H.~Hardy, M.~Riesz,
\emph{The general theory of Dirichlet's series},
Cambridge Tracts in Mathematics and Mathematical Physics, 1964.


\bibitem{helson} H.~Helson, \emph{Dirichlet series}, Regent Press, 2005.

\bibitem{Kahane} J.-P.~Kahane, \emph{Sur les fonctions moyenne-p\'eriodiques born\'ees}, Ann. Inst. Fourier \textbf{7} (1957), 293--314.

\bibitem{Landau}
E.~Landau,
\emph{\"Uber die gleichm\"a\ss ige Konvergenz Dirichletscher Reihen},
Math. Z. \textbf{11} (1921), 317--318.


\bibitem{Meyer} Y.~Meyer, \emph{Global and local estimates on trigonometric sums}, Trans. R. Norw. Soc. Sci. Lett. 2018(2), 1--25.

\bibitem{Neder} L.~Neder, \emph{Zum Konvergenzproblem der Dirichletschen Reihen beschr\"ankter Funktionen}, Math. Z. \textbf{14} (1922), 149--158.

\bibitem{schoolmann}
I.~Schoolmann,
\emph{On Bohr's theorem for general Dirichlet series},
Math. Nachr. \textbf{293} (2020), 1591--1612.










\end{thebibliography}
\end{document}